\documentclass{article}
\usepackage{latexsym,amsfonts,amsmath,amsthm,amssymb,makeidx}
\usepackage[title]{appendix}
\usepackage{CJK,CJKnumb,CJKulem,times,dsfont,ifthen,mathrsfs,latexsym,amsfonts, color}
\usepackage{amsmath,amsthm,makeidx,fontenc,amssymb,bm,graphicx,psfrag,listings, curves,extarrows}
\usepackage{cite}

\let\oldbibliography\thebibliography
\renewcommand{\thebibliography}[1]{%
\oldbibliography{#1}%
\setlength{\itemsep}{0pt}%
}

\makeindex
\newtheorem{definition}{Definition}[section]
\newtheorem{theorem}{Theorem}[section]
\newtheorem{lemma}{Lemma}[section]
\newtheorem{corollary}{Corollary}[section]
\newtheorem{proposition}{Proposition}[section]
\newtheorem{remark}{Remark}[section]

\newcommand{\bt}{\begin{theorem}}
\newcommand{\et}{\end{theorem}}
\newcommand{\bl}{\begin{lemma}}
\newcommand{\el}{\end{lemma}}
\newcommand{\bd}{\begin{definition}}
\newcommand{\ed}{\end{definition}}
\newcommand{\bc}{\begin{corollary}}
\newcommand{\ec}{\end{corollary}}
\newcommand{\bp}{\begin{proof}}
\newcommand{\ep}{\end{proof}}
\newcommand{\bx}{\begin{example}}
\newcommand{\ex}{\end{example}}
\newcommand{\bi}{\begin{exercise}}
\newcommand{\ei}{\end{exercise}}
\newcommand{\bo}{\begin{prop}}
\newcommand{\eo}{\end{prop}}
\newcommand{\br}{\begin{remark}}
\newcommand{\er}{\end{remark}}
\newcommand{\be}{\begin{equation}}
\newcommand{\ee}{\end{equation}}
\newcommand{\ba}{\begin{align}}
\newcommand{\ea}{\end{align}}
\newcommand{\bn}{\begin{enumerate}}
\newcommand{\en}{\end{enumerate}}
\newcommand{\bg}{\begin{align*}}
\newcommand{\bcs}{\begin{cases}}
\newcommand{\ecs}{\end{cases}}

\newcommand{\bean}{\begin{eqnarray*}}
\newcommand{\eean}{\end{eqnarray*}}

\numberwithin{equation}{section}

\begin{document}
\title{{\bf Qualitative analysis for an elliptic system in the punctured space}\thanks {Supported by  NSFC.   \qquad E-mail addresses:  hui-yang15@mails.tsinghua.edu.cn (H.  Yang),   \qquad \qquad \qquad  wzou@math.tsinghua.edu.cn (W.  Zou)}
}

\date{}
\author{\\{\bf Hui  Yang$^{1}$,\;\;  Wenming Zou$^{2}$}\\
\footnotesize {\it  $^{1}$Yau Mathematical Sciences Center, Tsinghua University, Beijing 100084, China}\\
\footnotesize {\it  $^{2}$Department of Mathematical Sciences, Tsinghua University, Beijing 100084, China}
}

\maketitle
\begin{center}
\begin{minipage}{110mm}
\begin{center}{\bf Abstract}\end{center}
In this paper, we investigate  the qualitative properties of  positive solutions for the following  two-coupled elliptic system in the punctured space:
\begin{equation*}
\begin{cases}
-\Delta u =\mu_1 u^{2q+1} + \beta u^q v^{q+1} \\
-\Delta v =\mu_2 v^{2q+1} + \beta v^q u^{q+1}
\end{cases} \textmd{in} ~\mathbb{R}^n \backslash \{0\},
\end{equation*}
where $\mu_1, \mu_2$ and $\beta$ are all positive constants, $n\geq 3$.  We establish a monotonicity formula that completely characterizes  the singularity of positive solutions.  We prove a sharp global estimate for both  components of positive solutions.  We also prove the nonexistence of positive semi-singular  solutions, which means  that one component is  bounded near the singularity and the other component is  unbounded near the singularity.

\vskip0.10in


\end{minipage}

\end{center}

\vskip0.390in

\section{Introduction and Main results}
In this paper, we investigate the  qualitative properties of  positive solutions for the following two-coupled elliptic system
\begin{equation}\label{SPS}
\begin{cases}
-\Delta u =\mu_1 u^{2q+1} + \beta u^q v^{q+1} \\
-\Delta v =\mu_2 v^{2q+1} + \beta v^q u^{q+1}
\end{cases} \textmd{in} ~\mathbb{R}^n \backslash \{0\},
\end{equation}
where $\mu_1, \mu_2$ and $\beta$ are all positive constants, $n\geq 3$ and $p:=2q+1>1$.  We say that $(u, v)$ is a positive solution of \eqref{SPS} if $u, v >0$ in $\mathbb{R}^n \backslash \{0\}$,  and that $(u, v)$ is a nonnegative solution of \eqref{SPS} if $u, v \geq 0$ in $\mathbb{R}^n \backslash \{0\}$.

\vskip0.1in

System \eqref{SPS} is related to the following nonlinear Schr\"{o}dinger system
\begin{equation}\label{Sch}
\begin{cases}
-\Delta u + \lambda_1 u =\mu_1 u^{2q+1} + \beta u^q v^{q+1} ~~~ \textmd{in} ~\Omega,\\
-\Delta v + \lambda_2 v =\mu_2 v^{2q+1} + \beta v^q u^{q+1}  ~~~ \textmd{in} ~\Omega,\\
u, v >0  ~~~\textmd{in} ~\Omega,
\end{cases}
\end{equation}
where $\Omega \subset \mathbb{R}^n$ is a domain. In the case $p=3$ $(q=1)$, the cubic system \eqref{Sch} arises in mathematical models from various physical phenomena, such as nonlinear optics and Bose-Einstein condensation. We refer the reader for these to the survey articles \cite{P-1,P-2}, which also contain information about the physical relevance of non-cubic nonlinearities. In the subcritical case $p<\frac{n+2}{n-2}$, the system \eqref{Sch} on a bounded smooth domain or on  $\mathbb{R}^N$ has   been widely investigated by variational methods and topological  methods in the last decades, see \cite{AC,Dancer10,Lin-Wei1,Maia,Quittner,Sirakov,Tavares} and references therein. In the critical case $p=\frac{n+2}{n-2}$,  the system \eqref{Sch} on a bounded  domain  has  been investigated in \cite{CL15,Chen12,Chen15}.

\vskip0.1in

In a recent paper \cite{CL}, Chen and Lin studied the  positive singular solutions of system \eqref{SPS} with the   critical  case $p=\frac{n+2}{n-2}$. They proved a sharp result about the  removable singularity and the nonexistence of positive  semi-singular solutions.   In this paper, we are concerned with these problems in the subcritical case $p < \frac{n+2}{n-2}$.

\vskip0.1in

When $1 < p \leq \frac{n}{n-2}$, by a Liouville type theorem in \cite{B-P} (also see \cite{S-Z}), we know that the system \eqref{SPS} has only the trivial nonnegative solution $u=v\equiv0$ in $\mathbb{R}^n \backslash \{0\}$.
Hence, we just need to consider the case $\frac{n}{n-2} < p < \frac{n+2}{n-2}$.

\vskip0.11in

In order to motivate our results on \eqref{SPS}, we first recall the classical results for the single elliptic equation
\begin{equation}\label{SE}
-\Delta u =u^p, ~~~~~~~  u>0 ~~~\textmd{in} ~\mathbb{R}^n \backslash \{0\}.
\end{equation}
The asymptotic behaviors  of solutions of equation \eqref{SE} near  0 and near $\infty$ were  studied in  \cite{CGS,Gidas2} for $\frac{n}{n-2} < p < \frac{n+2}{n-2}$ and in \cite{CGS,Fowler,K-M-P-S}  for $p=\frac{n+2}{n-2}$.  In the other case of $p$,  see  Lions \cite{L} for $1 < p < \frac{n}{n-2}$, Aviles \cite{A} for $p=\frac{n}{n-2}$ and  Bidaut-V\'eron and V\'eron \cite{B-V} for $p > \frac{n+2}{n-2}$, where the local  behavior of solutions of equation \eqref{SE} in a punctured ball was studied.  
When $\frac{n}{n-2} < p < \frac{n+2}{n-2}$, it is well know that the function
\begin{equation}\label{SE-01}
u(x)=C_0 |x|^{-\frac{2}{p-1}}
\end{equation}
with
\begin{equation}\label{SE-02}
C_0=\left[ \frac{2(n-2)}{(p-1)^2} \left( p-\frac{n}{n-2} \right) \right]^{1/(p-1)}
\end{equation}
is an explicit singular solution for \eqref{SE}.  Besides \eqref{SE-01}, there exists another (one-parameter family) solution of \eqref{SE} which satisfies
\begin{equation}\label{SE-03}
u(x) \sim
\begin{cases}
C_0 |x|^{-\frac{2}{p-1}}~~~~ & \textmd{near}~ x=0, \\
\lambda |x|^{-(n-2)}~~~~ & \textmd{near}~ x=+\infty,
\end{cases}
\end{equation}
where $C_0$ is given by \eqref{SE-02} and $\lambda >0$. See Appendix A in \cite{Gidas2}.  On the other hand, we see easily that both   components of the positive solution $(u, v)$ of \eqref{SPS} satisfy
\begin{equation}\label{SEI}
-\Delta w \geq w^p, ~~~~~~~  w>0 ~~~\textmd{in} ~\mathbb{R}^n \backslash \{0\}
\end{equation}
up to a multiplication. It follows from Corollary II in \cite{S-Z} that the inequality \eqref{SEI} has a positive solution $w_0\in C^2(\mathbb{R}^n)$ if $p>\frac{n}{n-2}$.  Furthermore, $w_0$ satisfies $-\Delta w_0 \geq w_0^p$ in entire space  $\mathbb{R}^n$.

\vskip0.1in

The purpose of the current paper is to study the positive solutions of system \eqref{SPS} with $\displaystyle \frac{n}{n-2} < p <\frac{n+2}{n-2}$. For any nonnegative  solution $(u, v)$ of \eqref{SPS}, we can prove that $\displaystyle \lim_{|x|\rightarrow 0} u(x)$ and $\displaystyle \lim_{|x|\rightarrow 0} v(x)$ make sense (the limit may be $+\infty$), see Theorem 2.1 in Sect. 2. Hence, there are three possibilities in general:

\begin{itemize}

\item [(1)] both $\displaystyle \lim_{|x|\rightarrow 0} u(x) < +\infty$ and $\displaystyle \lim_{|x|\rightarrow 0} v(x) < +\infty$, i.e., $(u, v)$ is an entire solution of \eqref{SPS} in $\mathbb{R}^n$.
By a Liouville type theorem, cf. Theorem 2  in \cite{RZ},  $u=v \equiv 0$ in this case.

\item [(2)] $\displaystyle \lim_{|x|\rightarrow 0} u(x) = \lim_{|x|\rightarrow 0} v(x)= +\infty$, i.e., the origin is a non-removable singularity of both $u$ and $v$, and we call solutions of this type {\it positive both-singular solutions at 0}. 

\item [(3)] either $$\lim_{|x|\rightarrow 0} u(x) < \lim_{|x|\rightarrow 0} v(x) =+\infty$$ or $$\lim_{|x|\rightarrow 0} v(x) < \lim_{|x|\rightarrow 0} u(x) =+\infty,$$ and we call solutions of this type {\it positive semi-singular solutions at 0}.

\end{itemize}

\vskip0.1in

The paper  \cite{CL} only studied the critical problem, where the Pohozaev identity  plays a very important role in characterizing the  positive singular solutions. More precisely, let $(u, v)$ be a positive solution of \eqref{SPS} with $p=\frac{n-2}{n+2}$. Denote $B_r:=\{ x\in \mathbb{R}^n : |x| < r\}$. By the Pohozaev identity, one has   $K(r; u, v)=K(s; u, v)$ for $0 < s < r$, where
$$
\aligned
K(r; u, v):= &   \int_{\partial B_r} \bigg[ \frac{n-2}{2}\left( u\frac{\partial u}{\partial \nu} + v\frac{\partial v}{\partial \nu} \right) -\frac{r}{2} (|\nabla u|^2 + |\nabla v|^2)  \\
&~~~   + r \left| \frac{\partial u}{\partial \nu} \right|^2  + r \left| \frac{\partial v}{\partial \nu} \right|^2 + \frac{r}{2^*} \left(\mu_1 u^{2^*} + \mu_2 v^{2^*} + 2\beta u^{\frac{2^*}{2}}  v^{\frac{2^*}{2}}\right) \bigg] d\sigma,
\endaligned
$$
and $\nu$ is the unit outer normal of $\partial B_r$. Hence,  $K(r; u,v)$ is a constant independent of $r$, and  we denote this constant by $K(u, v)$. Then the result of removable singularity in \cite{CL} is as follows.

\vskip0.1in

\noindent {\bf Theorem A}   \cite{CL}  {\it Let $\mu_1, \mu_2, \beta >0$ and $(u, v)$ be a nonnegative solution of \eqref{SPS} with $p=\frac{n-2}{n+2}$.  Then $K(u,v) \leq 0$. Furthermore, $K(u, v)=0$ if and only if $u, v \in C^2(\mathbb{R}^n)$, namely both $u$ and $v$ are smooth at 0. }

\vskip0.1in

\noindent However, there are some differences between the  subcritical case and the  critical case.  The classical Pohozaev identity seems to be unavailable  to characterize the  positive singular solutions in subcritical case.

\vskip0.1in

Here we will establish a monotonicity formula to character  the  positive singular solutions of \eqref{SPS} in the case $\displaystyle \frac{n}{n-2} < p < \frac{n+2}{n-2}$.  For $r>0$, define
\begin{equation}\label{MF}
\aligned
E(r; u, v):= &   r^{\tau}\int_{\partial B_r} \bigg[\frac{2}{p-1}\left( u\frac{\partial u}{\partial \nu} + v\frac{\partial v}{\partial \nu} \right) -\frac{r}{2} (|\nabla u|^2 + |\nabla v|^2) \\
&~~~ + r \left| \frac{\partial u}{\partial \nu} \right|^2  + r \left| \frac{\partial v}{\partial \nu} \right|^2  + \frac{\tau}{r(p-1)} (u^2 + v^2)  \bigg] d\sigma  \\
&~~~  + r^{\tau}\int_{\partial B_r}  \frac{r}{p+1} \left(\mu_1 u^{p+1} + \mu_2 v^{p+1} + 2\beta u^{q+1}  v^{q+1}\right) d\sigma,
\endaligned
\end{equation}
where
$
\tau=\frac{4}{p-1} - n +2.
$
Let $(u, v)$ be a nonnegative solution of \eqref{SPS} with $\displaystyle \frac{n}{n-2} < p < \frac{n+2}{n-2}$. Then we will prove that  $E(r; u, v)$ is nondecreasing and bounded for  $r\in (0, +\infty)$ (see Proposition \ref{P-01} and Lemma \ref{L-02} in Sect. 3).  Hence, we may define  the limits
$$
E(0;u,v):=\lim_{r \rightarrow 0^+} E(r; u, v)~~~~ \textmd{and}~~~~E(\infty;u,v):=\lim_{r \rightarrow +\infty} E(r; u, v).
$$
Our first result is about the singularity of positive solutions.  We partly classify the nonnegative solutions of \eqref{SPS} by  $E(r; u, v)$.
\begin{theorem}\label{T01}
Let $\mu_1, \mu_2, \beta >0$ and $(u, v)$ be a nonnegative solution of \eqref{SPS} with $\frac{n}{n-2} < p <\frac{n+2}{n-2}$.  Then
$$\Big\{ E(0;u,v), E(\infty;u,v) \Big\}  \subset \left\{ 0, -\frac{p-1}{2(p+1)} (k^2 + l^2) C_0^{p+1} \right\},$$
where $C_0$ is given by \eqref{SE-02} and $k, l \geq 0 $ (not all 0)  satisfy
\begin{equation}\label{T011}
\mu_1 k^{2q} + \beta k^{q-1} l^{q+1} =1,~~~~\mu_2 l^{2q} + \beta l^{q-1} k^{q+1} =1.
\end{equation}
Furthermore, we get

\begin{itemize}

\item [(1)] $E(0;u,v)=0$ if and only if $(u, v)$ is trivial, i.e., $u=v\equiv0$.

\item [(2)] $E(r; u, v)\equiv constant\not=0 $ if and only if  $(u, v)$ is of the form
\begin{equation}\label{form}
u(x)=k C_0 |x|^{-\frac{2}{p-1}},~~~ v(x)=l C_0 |x|^{-\frac{2}{p-1}},
\end{equation}
where $C_0$ is given by \eqref{SE-02} and   $k, l \geq 0 $ (not all 0)   satisfy  \eqref{T011}.

\end{itemize}

\end{theorem}

%Remark that the similar monotonicity formulas  have established in recent papers \cite {GKS,Y-Z} to study the isolated singularities of some semilinear elliptic problems.

\begin{remark}\label{R-01}
We point out that if one of $k$ and $l$ is 0, such as $k=0$, then system of equations \eqref{T011}  is understood as  a single equation $\mu_2 l^{2q}=1$.
\end{remark}

\begin{remark}\label{R-02}
It is easy to see that the non-negative solution of \eqref{T011}, which is not all 0,  is not unique.   This fact determines that the behavior of positive solutions of system \eqref{SPS} is much more complicated than that of  single equation \eqref{SE}.
\end{remark}

From now on, we denote
$$
A_{k,l}= \frac{p-1}{2(p+1)} (k^2 + l^2) C_0^{p+1}.
$$
Theorem \ref{T01} shows that if $(u, v)$ is a positive solution of \eqref{SPS} with $\frac{n}{n-2} < p < \frac{n+2}{n-2}$,  then it is either  both-singular  or  semi-singular at 0, and $E(0;u,v)=- A_{k,l}<0$.

\vskip0.13in

Our second result is concerned with the asymptotic behaviors of  positive solutions near 0 and near  $\infty$.  Just like the critical case in \cite{CL}, a basic open question is {\it whether positive semi-singular  solutions at 0 exist or not}.    Since the system \eqref{SPS} with $p=\frac{n+2}{n-2}$ is conformally invariant,  the behaviors of positive singular  solutions near  0 and near $\infty$ is equivalent, see \cite{CL}.  For the subcritical case $\frac{n}{n-2} < p < \frac{n+2}{n-2}$,  the system \eqref{SPS} is not conformally invariant,  the behaviors of positive singular  solutions near  0 and near $\infty$ is very different.

\vskip0.10in

Let $u$ be a positive solution of equation \eqref{SE} with $\frac{n}{n-2} < p <\frac{n+2}{n-2}$, by Theorem 3.6 in \cite{Gidas2}, then either the singularity at $\infty$ is removable, i.e., there exists a constant $c >0$ such that
\begin{equation}\label{SE1-19}
u(x)  \leq c |x|^{-(n-2)},~~~~ |x|\geq 1,
\end{equation}
or there exist constants $0< c_1\leq c_2$ such that
\begin{equation}\label{SE2-28}
c_1 |x|^{-\frac{2}{p-1}} \leq u(x) \leq c_2 |x|^{-\frac{2}{p-1}},~~~~|x|\geq 1.
\end{equation}
In order to describe the behavior of positive solutions  of system \eqref{SPS} near $\infty$,  we  introduce the following definition. 

\begin{definition}
Assume $\frac{n}{n-2} < p < \frac{n+2}{n-2}$. we  say that   $(u, v)$ is a positive semi-singular   solution of \eqref{SPS} at $\infty$ if either 
$$ 
u(x)\leq C |x|^{-(n-2)} ~~~ \textmd{and} ~~~  C_1|x|^{-\frac{2}{p-1}} \leq v(x) \leq C_2 |x|^{-\frac{2}{p-1}} ~~ \textmd{near}~ \infty,
$$
or 
$$ 
v(x)\leq C |x|^{-(n-2)} ~~~ \textmd{and} ~~~   C_1|x|^{-\frac{2}{p-1}} \leq u(x) \leq C_2 |x|^{-\frac{2}{p-1}}~~ \textmd{near}~  \infty.
$$ 
\end{definition}
A new question is {\it whether positive semi-singular solutions of \eqref{SPS} at $\infty$ exist or not}. 

\begin{theorem}\label{T02}
Let $\mu_1, \mu_2, \beta >0$ and $\frac{n}{n-2} < p < \frac{n+2}{n-2}$.
\begin{itemize}
\item [(1)]  System \eqref{SPS} has no positive semi-singular  solutions at 0 whenever $n \geq 4$.

\item [(2)]  Assume $n \geq 3$. In particular, if $n=3$, we suppose additionally that the system \eqref{SPS} has no positive semi-singular  solutions at 0.  \\
Then, for any positive solution $(u, v)$ of \eqref{SPS},  there exist constants $0< C_1\leq C_2$ such that
\begin{equation}\label{T021}
C_1 |x|^{-\frac{2}{p-1}} \leq u(x), ~ v(x) \leq C_2 |x|^{-\frac{2}{p-1}},~~~~x\in B_1 \backslash \{0\}.
\end{equation}

\item [(3)] System \eqref{SPS} has no positive semi-singular  solutions at $\infty$ whenever $n \geq 4$.

\item [(4)] Assume $n \geq 3$. In particular, if $n=3$, we suppose additionally that the system \eqref{SPS} has no positive semi-singular  solutions at $\infty$. \\
Then, for any positive solution $(u, v)$  of \eqref{SPS},  either the singularity at $\infty$ is removable, i.e., there exists $C_3>0$ such that
\begin{equation}\label{T022}
u(x), ~ v(x) \leq C_3 |x|^{-(n-2)},~~~~ |x|\geq 1,
\end{equation}
or there exist constants $0< C_1\leq C_2$ such that
\begin{equation}\label{T023}
C_1 |x|^{-\frac{2}{p-1}} \leq u(x), ~ v(x) \leq C_2 |x|^{-\frac{2}{p-1}},~~~~|x|\geq 1.
\end{equation}

\end{itemize}
\end{theorem}

Theorem \ref{T02} gives a classification of positive solutions to system \eqref{SPS} whenever $n\geq 4$. In a subsequent work \cite{YZ1}, we will study the asymptotic symmetry and classification of isolated singularities of positive solutions to system \eqref{SPS} in a {\it punctured ball}.  We remark that above theorems are very important to study these problems in a punctured ball.   We also remark that the single equation \eqref{SE} $(\frac{n}{n-2} < p < \frac{n+2}{n-2})$  in a punctured ball or in punctured space was well studied in the seminal papers \cite{CGS, Gidas2}.

%, our work here and the subsequent \cite{YZ1} can be seen as generalizations  of \cite{CGS, Gidas2} to systems.

\vskip0.13in

Theorem \ref{T02} shows that, for $n\geq 4$, 0 (or $\infty$) is a non-removable singularity of $u$ if and only if 0  (resp. $\infty$) is a non-removable singularity of $v$, but we don't know whether this conclusion holds or not for $n=3$.  We tend to believe that it also holds for $n=3$, by which we obtain the sharp estimate of positive solutions. These are some essential differences between the system \eqref{SPS} and the scalar equation \eqref{SE}.

\vskip0.1in

Remark that, in Theorem \ref{T02}, we get the {\it sharp} estimates for  {\it both  components} of  positive solutions.  In a very recent paper \cite{GKS}, Ghergu, Kim and Shahgholian  studied the positive singular solutions of system
\begin{equation}\label{GKS}
-\Delta {\bf u}=|{\bf u}|^{p-1} {\bf u}
\end{equation}
in a punctured ball, where ${\bf u}=(u_1, \cdots, u_m)$. When $\frac{n}{n-2} < p < \frac{n+2}{n-2}$, they only obtain the similar sharp estimate for $|{\bf u}(x)|$ as a whole  {\it rather than} each component $u_i$ of the  positive vector solutions ${\bf u}$.     We also point out that, the sharp estimates of singular  solutions for the scalar equation \eqref{SE} is relatively  simple. However, coupled system \eqref{SPS} turns out to be much more delicate  and complicated than \eqref{SE}.  Further, in the first equation of \eqref{SPS}, the power of $u$ in the coupling term is $q$, which is $>$1 if $n=3$ and $<1$ if $n \geq 4$. This fact makes the argument depending heavily on the dimension. This is another  important  difference between \eqref{SPS} and \eqref{SE}.

\vskip0.1in

Assume $\frac{n}{n-2} < p < \frac{n+2}{n-2}$ and let $W$ be a  positive  solution of \eqref{SE}, then $(kW, lW)$ is a positive both-singular  solution of \eqref{SPS}, where $k, l>0$ satisfy \eqref{T011}. Conversely,  there is an  interesting problem remaining: {\it whether any  positive solutions are of the form $(kW, lW)$, where $W$ is a  solution of \eqref{SE}}?  This question seems very tough.   We remark that there are some papers to study the proportionality of components ( $u/v\equiv $ constant)  of  {\it entire} positive solutions for others systems, such as see \cite{Chen-Li,F,MSS,QS},  but we can not obtain the conclusion $u/v\equiv constant$ to system \eqref{SPS} via the ideas of these papers, since \eqref{SPS} does not satisfy the structural conditions  in these papers, and system \eqref{SPS} has an isolated singularity 0.
Therefore, system \eqref{SPS} can not be reduced to a single equation. We will prove the radial symmetry of positive singular  solutions of \eqref{SPS} and then prove our result by analyzing an ODE system, which turns out to be very delicate   and complicated.  Remark that  the ODE system corresponding to the subcritical case in this paper  is very different from the critical case in \cite{CL},  and some new observations and ideas are needed.  In particular, a very important monotonicity property of ODE system in \cite{CL}  does not exist in our problem.  See Sect. 3 for a detailed explanation.

\vskip0.1in

The rest of this paper is structured as follows. Sect. 2 is devoted to prove that positive singular solutions of \eqref{SPS} are radially symmetric via the method of moving planes.    In Sect. 3,  we establish the monotonicity formula and some crucial lemmas. Theorem \ref{T01} and \ref{T02} are proved in Sect. 4.   We will see that our arguments are much more delicate than the scalar equation  and are also very different with  the critical system.  We will denote positive constants (possibly different in different lines) by $C, \bar{C}, C_1, C_2, \cdots$.

\section{Radial symmetry}
In this section we apply the method of moving planes to prove the radial symmetry of positive singular solutions, which extend the classical result in Caffarelli-Gidas-Spruck  \cite{CGS} to system \eqref{SPS}. We may also see \cite{CL} for this extension to \eqref{SPS} with $p=\frac{n+2}{n-2}$.  Since we will use the Kelvin transform, it does not seem trivial to get the strictly decreasing of $u(r)$ and $v(r)$ for $r=|x|$ from this proof. We will prove  this strictly decreasing property via some estimates in Lemma \ref{L-03} in Sect.  3.
\begin{theorem}\label{R-S}
Let $(u, v)$ be a  positive solution  of \eqref{SPS} with $\frac{n}{n-2} < p \leq \frac{n+2}{n-2}$. Assume that $\limsup_{|x| \to 0} u(x) =+\infty$ or $\limsup_{|x| \to 0} v(x) =+\infty$. Then both $u$ and $v$ are radially symmetry about the origin and are strictly decreasing about $r=|x|>0$.
\end{theorem}
\bp
We follow the idea  in  \cite{CGS} to use the method of moving planes.
Without loss of generality, we assume that $\limsup_{|x| \to 0} u(x) =+\infty$.
Fix an arbitrary point $z\neq 0$ and define the Kelvin transform
$$
U(x)=\frac{1}{|x|^{n-2}}u \left(z + \frac{x}{|x|^2} \right),~~~~~ V(x)=\frac{1}{|x|^{n-2}}v \left(z + \frac{x}{|x|^2} \right).
$$
Then $(U, V)$ satisfies
\begin{equation}\label{SPS0-11}
\begin{cases}
-\Delta U =|x|^{\alpha} \left( \mu_1 U^{2q+1} + \beta U^q V^{q+1}\right) \\
-\Delta V =|x|^{\alpha} \left( \mu_2 V^{2q+1} + \beta V^q U^{q+1} \right)
\end{cases} \textmd{in} ~\mathbb{R}^n \backslash \{0, z_0\},
\end{equation}
where $\alpha=p(n-2) - (n+2)$ and $z_0=-z/|z|^2$.  Clearly, $U$ and $V$ are singular at 0 and $z_0$.  Define an axis going through 0 and $z$, we shall show that both $U$ and $V$  are axisymmetric.   To this end, we consider any reflection direction $\eta$ orthogonal to this axis. We may assume, without loss of generality,  that $\eta=(0, \cdots, 0, 1)$ is the positive $x_n$ direction.   For any $\lambda >0$, let
$$
\Sigma_\lambda=\{ x\in \mathbb{R}^n : x_n > \lambda\},~~~~~~T_\lambda=\{ x\in \mathbb{R}^n : x_n = \lambda\}.
$$
We denote $x^\lambda=(x^\prime, 2\lambda-x_n)$ as the reflection of the point $x=(x^\prime, x_n)$ about  the plane $T_\lambda$.
Since $U$ and $V$ have the harmonic asymptotic expansion  (see (2.6) in \cite{CGS}) at infinity, by Lemma 2.3 in \cite{CGS}, there exist large positive constants $\lambda_0>10$ and  $R>|z_0| +10$ such that for any $\lambda \geq \lambda_0$, we have
\begin{equation}\label{PT21}
U(x) < U(x^\lambda),~~~~~~ V(x) < V(x^\lambda)~~~~~~~ \textmd{for} ~~ x\in \Sigma_\lambda,~~~~ |x^\lambda| >R.
\end{equation}
By the maximum principle for super harmonic functions with isolated singularities, cf. Lemma 2.1 in \cite{Che-L}, there exists $C>0$ such that
\begin{equation}\label{PT22}
U(x), V(x) \geq C~~~~~~~ \textmd{for} ~~ x\in \overline{B_R} \backslash \{0, z_0\}.
\end{equation}
Since $U(x), V(x) \to 0$ as $|x| \to +\infty$, by \eqref{PT21} and \eqref{PT22}, there exists $\lambda_1 > \lambda_0$ such that for any $\lambda \geq \lambda_1$, we have
\begin{equation}\label{PT23}
U(x) \leq U(x^\lambda),~~~~~~ V(x) \leq V(x^\lambda)~~~~~~~ \textmd{for} ~~ x\in \Sigma_\lambda,~~~~ x^\lambda \notin  \{0, z_0\}.
\end{equation}
Define
$$
\lambda^*:=\inf \{ \bar{\lambda} >0 ~ |~ \eqref{PT23} ~ \textmd{holds} ~ \textmd{for} ~  \textmd{all} ~ \lambda \geq \bar{\lambda}\}.
$$
Suppose $\lambda^* >0$. Then \eqref{PT23} also holds for $\lambda=\lambda^*$. Since $\limsup_{|x| \to 0} u(x)=+\infty$, we get $\limsup_{x \to z_0} U(x)=+\infty$ and hence $U(x) \not\equiv  U(x^{\lambda^*})$.  Define  $W_1(x)=U(x^{\lambda^*})$ and $W_2(x)=V(x^{\lambda^*})$ for $x\in \Sigma_{\lambda^*}$ and $x^{\lambda^*} \notin  \{0, z_0\}$. Then we have
$$
\begin{cases}
-\Delta (W_1- U) \geq 0 \\
-\Delta (W_2 - V) \geq \beta |x|^{\alpha}  V^q \big( W_1^{q+1} -U^{q+1} \big)
\end{cases} \textmd{for} ~~ x\in \Sigma_{\lambda^*},~~ x^{\lambda^*} \notin  \{0, z_0\}.
$$
By  the strong maximum principle we obtain
\begin{equation}\label{PT24}
U(x) <  W_1(x),~~~~~~ V(x) <  W_2(x)~~~~~~~ \textmd{for} ~~ x\in \Sigma_{\lambda^*},~~~~ x^{\lambda^*} \notin  \{0, z_0\}.
\end{equation}
Note that  $0, z_0 \notin T_{\lambda^*}$ because of $\lambda^* >0$. By the Hopf boundary lemma, we have for any $x\in T_{\lambda^*}$,
\begin{equation}\label{PT25}
\frac{\partial ( W_1- U)(x)}{\partial x_n} = -2 \frac{\partial U(x)}{\partial x_n} >0,~~~~\frac{\partial (W_2 - V)(x)}{\partial x_n} = -2 \frac{\partial V(x)}{\partial x_n} >0.
\end{equation}
By the definition $\lambda^*$, there exists $\lambda_j \to \lambda^*$ $(\lambda_j < \lambda^*)$  such that \eqref{PT23} does not hold for $\lambda=\lambda_j$. Without loss of generality and up to a subsequence, we assume that there exists $x_j \in \Sigma_{\lambda_j}$ such that $U(x_j^{\lambda_j}) < U(x_j)$. It follows Lemma 2.4 in \cite{CGS} (the plane $x_n=0$ there corresponds to $x_n=\lambda^*$ here) that  $|x_j|$ are uniformly bounded. Hence, up to a subsequence, $x_j \to \bar{x}\in \overline{\Sigma_{\lambda^*}}$ with $U(\bar{x}^{\lambda^*}) \leq U(\bar{x})$.  By \eqref{PT24},  we have $\bar{x} \in T_{\lambda^*}$ and then $\frac{\partial U}{\partial x_n} (\bar{x}) \geq 0$, a contradiction with \eqref{PT25}. Hence $\lambda^*=0$ and so both $U$ and $V$ are axisymmetric about the axis going through 0 and $z$. Since $z$ is arbitrary, $u$ and $v$ are both axisymmetric about any axis going through 0, and so $u, v$ are both radially symmetry about the origin. The strictly decreasing property of $u$ and $v$ will be proved via some estimates in Lemma \ref{L-03} in Sect. 3.
\ep

\section{Monotonicity formula and crucial lemmas}
In this section, we  establish the monotonicity formula and some crucial lemmas.  Let $(u, v)$ be a positive solution of \eqref{SPS}. By Theorem \ref{R-S} in Sect. 2 and Theorem 2 in \cite{RZ},  we may assume that $u(x)=u(|x|)$ and $v(x)=v(|x|)$ are radially symmetric functions.  Denote $\delta=\frac{2}{p-1}$.
We use the classical change of variables as in Fowler \cite{Fowler}. Let $t=-\ln r$ and
$$
w_1(t):=r^{\delta} u(r)= e^{-\delta t} u(e^{-t}),~~~~~~~~ w_2(t):=r^{\delta} v(r)= e^{-\delta t} v(e^{-t}).
$$
Then,  by a direct  calculation  $(w_1, w_2)$ satisfies
\begin{equation}\label{S301}
\begin{cases}
w_1^{\prime\prime} + \tau w_1^{\prime} -  \sigma w_1 + \mu_1 w_1^{2q+1} + \beta w_1^q w_2^{q+1}=0 \\
w_2^{\prime\prime} + \tau w_2^{\prime} -  \sigma w_2 + \mu_2 w_2^{2q+1} + \beta w_2^q w_1^{q+1}=0
\end{cases} t\in \mathbb{R},
\end{equation}
where $w_i^{\prime}:=\frac{d w_i}{dt}$, $w_i^{\prime\prime}:=\frac{d^2 w_i}{dt^2}$ and
$$
\tau=\frac{4}{p-1} -n + 2,~~~~~ \sigma=\frac{2}{p-1}\left(n-2-\frac{2}{p-1}\right).
$$
We note that $\tau, \sigma >0$ due to $\frac{n}{n-2} < p < \frac{n+2}{n-2}$.

\vskip0.1in

When $p=\frac{n+2}{n-2}$, we see that $\tau=0$ and $\sigma=\delta^2=\left(\frac{n-2}{2}\right)^2$.  Consider two functions $f_1, f_2 : \mathbb{R} \rightarrow \mathbb{R}$ by
$$
f_1(t):=-\frac{1}{2}|w_1^\prime|^2 +\frac{\delta^2}{2}w_1^2 -\frac{\mu_1}{p+1} w_1^{p+1},~f_2(t):=-\frac{1}{2}|w_2^\prime|^2 +\frac{\delta^2}{2}w_2^2 -\frac{\mu_1}{p+1} w_2^{p+1}.
$$
Then by \eqref{S301} we have
$$
f_1^{\prime}(t)=\beta w_1^{q} w_2^{q+1} w_1^\prime,~~~~~f_2^{\prime}(t)=\beta w_2^{q} w_1^{q+1} w_2^\prime.
$$
That is, for $i=1,2$,  {\it the monotonicity of $f_i$ is exactly the same as the monotonicity of $w_i$}. This  monotonicity property plays a very important role and is used frequently in \cite{CL}.  However, when $\frac{n}{n-2} < p < \frac{n+2}{n-2}$, we have $\tau >0$,  there is {\it no similar monotonicity  property}. In addition, we note that when $\tau=0$, system \eqref{S301} is invariant under reflection but our problem \eqref{S301}  is {\it not} because $\tau>0$.
Therefore, we need some new observations and different  ideas to deal with our ODE system \eqref{S301}.

\vskip0.1in

Multiplying the first equation of \eqref{S301} by $w_1^{\prime}$, the second equation of \eqref{S301} by $w_2^{\prime}$, we easily obtain the following identity
\begin{equation}\label{S302}
\Psi^{\prime} (t)=  - \tau \left[ (w_1^{\prime})^2 + (w_2^{\prime})^2 \right](t),
\end{equation}
where
\begin{equation}\label{S303}
\aligned
\Psi  (t)= & \frac{1}{2} \left( |w_1^{\prime}|^2  + |w_2^{\prime}|^2 - \sigma \left( w_1^2 + w_2^2 \right) \right) (t) \\
 &~ +\frac{1}{p+1} \left( \mu_1 w_1^{p+1} + 2\beta w_1^{q+1} w_2^{q+1} + \mu_2 w_2^{p+1} \right) (t).
\endaligned
\end{equation}
Therefore,  we have $\Psi^{\prime} (t) \leq 0$ in $\mathbb{R}$, namely $\Psi(t)$ is decreasing for $t\in \mathbb{R}$.  Indeed, by an easy computation, we can deduce from \eqref{MF} that $E(r; u, v)= \sigma_{n-1} \Psi  (t)$, where $t=-\ln r$ and $\sigma_{n-1}$ is the area of the unit sphere in $\mathbb{R}^n$. Hence we have
\begin{equation}\label{S304}
\frac{d}{dr} E(r; u, v)= -\frac{\sigma_{n-1}}{r} \Psi^{\prime} (t).
\end{equation}
From this, we easily obtain the following monotonicity formula.
\begin{proposition}\label{P-01}
Let $(u, v)$ be a nonnegative solution of \eqref{SPS} with $\frac{n}{n-2} < p \leq \frac{n+2}{n-2}$. Then $E(r; u, v)$ is nondecreasing for $r\in (0, \infty)$. Moreover,
\begin{equation}\label{S305}
\frac{d}{dr} E(r; u, v)= \tau r^{\tau} \int_{\partial B_r} \bigg[ \bigg( \frac{\partial u}{\partial \nu} + \frac{2}{p-1} \frac{u}{r}\bigg)^2 + \bigg( \frac{\partial v}{\partial \nu} + \frac{2}{p-1} \frac{v}{r} \bigg)^2 \bigg],
\end{equation}
where $\nu$ is the unit outer normal of $\partial B_r$.
\end{proposition}

To prove  that $E(r; u, v)$ is bounded in $(0, +\infty)$,  we need the following estimates. There are many ways to prove them, for example, we can use the method in  Pol\'{a}\v{c}ik-Quittner-Souplet  \cite{Pol07},  and here we use a simple idea from  Lemma 4.2 in \cite{GKS}.

\begin{lemma}\label{L-01}
Let $(u, v)$ be a nonnegative solution of \eqref{SPS} with $\frac{n}{n-2} < p \leq \frac{n+2}{n-2}$. Then
\begin{equation}\label{S306}
u(x), v(x) \leq C |x|^{-\frac{2}{p-1}}~~~~ \textmd{in} ~ \mathbb{R}^n \backslash \{0\}
\end{equation}
and
\begin{equation}\label{S307}
|\nabla u(x)|, |\nabla v(x)| \leq C |x|^{-\frac{2}{p-1} -1}~~~~ \textmd{in} ~ \mathbb{R}^n \backslash \{0\},
\end{equation}
where both two constants $C$ depend only on $n, p, \mu_1, \mu_2$ and $\beta$.
\end{lemma}
\bp
By the strong maximum principle, we assume that $u>0$ in $\mathbb{R}^n \backslash \{0\}$. Since $u$ is superharmonic, it follows from Lemma 2.1 in \cite{Che-L} that
\begin{equation}\label{S308}
\liminf_{x\to 0} u(x) >0.
\end{equation}
Let $\widetilde{u}=u^{1-p}$. Then $\widetilde{u}$ satisfies
$$
\Delta \widetilde{u} \geq \frac{p}{p-1} \frac{ |\nabla \widetilde{u}| }{ \widetilde{u} }  + \mu_1 (p-1),~~~~ \textmd{in} ~ \mathbb{R}^n \backslash \{0\}.
$$
Hence, the auxiliary function
$$
\widetilde{w}(x)=\widetilde{u}(x) - \frac{\mu_1(p-1)}{2n} |x|^2
$$
is subharmonic in $\mathbb{R}^n \backslash \{0\}$.  By \eqref{S308}, $\widetilde{w}$ is bounded near the origin.  Thus, for any $r>0$,  it follows from Theorem 1 in \cite{G-S} that
$$
0 \leq \limsup_{x\to 0} \widetilde{w}(x) \leq \sup_{\partial B_r} \widetilde{w} = \sup_{\partial B_r} \widetilde{u} - \frac{\mu_1(p-1)}{2n} r^2.
$$
In terms of $u$, we have
$$
\inf_{\partial B_r} u \leq \left(\frac{\mu_1 (p-1)}{2n}\right)^{-\frac{1}{p-1}} r^{-\frac{2}{p-1}}.
$$
By the radial symmetry obtained in Theorem \ref{R-S},  we get the estimate of $u$ in \eqref{S306}. Using a similar argument, we can  get the estimate of $v$ in \eqref{S306}.

\vskip0.1in

For any $x_0 \in \mathbb{R}^n \backslash \{0\}$,  take $\lambda =\frac{|x_0|}{2}$ and define
$$
u_1(x)=\lambda^{\frac{2}{p-1}} u(x_0 + \lambda x),~~~~ v_1(x)=\lambda^{\frac{2}{p-1}} v(x_0 + \lambda x).
$$
Then $(u_1, v_1)$ satisfies the system \eqref{SPS} in $B_1$. By \eqref{S306}, $|u_1|, |v_1| \leq C$ in $B_1$. Standard gradient estimate gives
$$
|\nabla u_1 (0)|, ~ |\nabla v_1 (0)| \leq C.
$$
Rescaling back we get \eqref{S307}.
\ep

By the above lemma, one easily obtains the boundedness of $E(r; u, v)$ for $r \in (0, +\infty)$.  Let's omit the proof.
\begin{lemma}\label{L-02}
Let $(u, v)$ be a nonnegative solution of \eqref{SPS} with $\frac{n}{n-2} < p \leq \frac{n+2}{n-2}$. Then $E(r; u, v)$ is uniformly bounded for all $r \in (0, +\infty)$.
\end{lemma}

Another   consequence  of  the upper bound \eqref{S306} is the following Harnack inequality. 

\begin{lemma}\label{har}
Let $(u, v)$ be a nonnegative solution of \eqref{SPS} with $\frac{n}{n-2} < p \leq \frac{n+2}{n-2}$.  Then for all $r>0$, we have
\begin{equation}\label{har-01}
\sup_{B_{2r} \backslash B_r} (u + v) \leq C \inf_{B_{2r} \backslash B_r} (u + v),
\end{equation}
where $C$ depends only on $n, p, \mu_1, \mu_2$ and $\beta$.
\end{lemma}
\bp
Let $z=u+v$, then $w$ satisfies the scalar equation
$$
-\Delta z= C(x) z^p~~~~~ \textmd{in}~ \mathbb{R}^n \backslash \{0\},
$$
where $C(x)$  is a bounded function for all $x\in \mathbb{R}^n \backslash \{0\}$.  For any $r>0$, let $z_r(x)=:z(rx)$, then $z_r$ satisfies
$$
-\Delta z_r= a_r(x) z_r ~~~~~ \textmd{in}~ \mathbb{R}^n \backslash \{0\},
$$
where $a_r(x):= r^2 C(rx)z^{p-1}(rx)$.  By Lemma \ref{L-01}, we have
$$
|a_r(x)| \leq  C |x|^{-2} \leq 4C~~~~~ \textmd{for}~  \frac{1}{2} \leq |x| \leq 4,
$$
where $C$ is a constant that depends only on $n, p, \mu_1, \mu_2$ and $\beta$.  Therefore, the classical Harnack  inequality gives
$$
\sup_{1\leq |x| \leq 2} z_r \leq C \inf_{1 \leq |x| \leq 2} z_r.
$$
Consequently, we get
$$
\sup_{B_{2r} \backslash B_r} (u + v) \leq C \inf_{B_{2r} \backslash B_r} (u + v).
$$
\ep

Now we prove the strictly decreasing property of positive solutions of \eqref{SPS}.
\begin{lemma}\label{L-03}
Let $(u, v)$ be a  positive solution of \eqref{SPS} with $\frac{n}{n-2} < p \leq \frac{n+2}{n-2}$. Then both $u^{\prime}(r) < 0$ and $v^{\prime}(r) < 0$ for all $r>0$.
\end{lemma}
\bp
By the divergence theorem, for any $0 < \epsilon < r $, we have
\begin{equation}\label{S309}
-\int_{\partial B_r} \frac{\partial u}{\partial \nu} + \int_{\partial B_\epsilon} \frac{\partial u}{\partial \nu} = \int_{B_r \backslash B_\epsilon} \mu_1 u^p + \beta u^q v^{q+1},
\end{equation}
where $\nu$ is the unit outer normal of $\partial B_r$ and $\partial B_\epsilon$. By Lemma \ref{L-01},
$$
\int_{\partial B_\epsilon} \left | \frac{\partial u}{\partial \nu} \right| \leq C \epsilon^{n-\frac{2}{p-1} - 2} \to 0 ~~~~~ \textmd{as} ~ \epsilon \to 0.
$$
Thus, letting $\epsilon \to 0$ in \eqref{S309}, we get
$$
- \sigma_{n-1} u^\prime (r)  r^{n-1} = \int_{B_r} \mu_1 u^p + \beta u^q v^{q+1} >0,
$$
and hence $u^\prime (r) <0$ for all $r>0$.  Similarly, we also have $v^\prime (r) <0$ for all $r>0$.
\ep

As a result of Lemma \ref{L-03}, we have the following important corollary.

\begin{corollary}\label{C-01}
Let $(w_1, w_2)$ be a positive solution of \eqref{S301} with $\frac{n}{n-2} < p \leq \frac{n+2}{n-2}$. Then $w_i^\prime (t) > -\delta w_i(t)$ for all $t\in \mathbb{R}$ and $i=1, 2$.
\end{corollary}
\bp
Since $w_1(t)=r^\delta u(r)$ with $r=e^{-t}$,  Lemma \ref{L-03} gives
$$
w_1^\prime (t)=-(\delta r^\delta u(r) + r^{\delta +1} u^{\prime} (r)) > -\delta w_1(t).
$$
Similarly, we can get $w_2^\prime (t) > -\delta w_2(t)$ for all $t\in \mathbb{R}$.
\ep

\section{Positive solutions}
In this section,  we prove Theorems \ref{T01} and \ref{T02}.

\vskip0.1in

\noindent {\bf Proof of Theorem \ref{T01}.} {\it Step 1.} We first prove that $E(r; u, v)\equiv constant $ if and only if  $(u, v)$ is of the form \eqref{form}.

\vskip0.1in

If $E(r; u, v)$ is a constant for all $r>0$.  By Proposition \ref{P-01},
$$
\left( \frac{\partial u}{\partial \nu} + \frac{2}{p-1} \frac{u}{r}\right)^2 + \left( \frac{\partial v}{\partial \nu} + \frac{2}{p-1} \frac{v}{r}\right)^2 =0~~~~ \textmd{in} ~ \mathbb{R}^n \backslash \{0\}.
$$
Integrating in $r$ we get
$$
u(x)=|x|^{-\frac{2}{p-1}} u\left( \frac{x}{|x|} \right) ~~ \textmd{and} ~~ v(x)=|x|^{-\frac{2}{p-1}} v \left( \frac{x}{|x|} \right).
$$
This shows that $(u, v)$  is homogeneous of degree $-\frac{2}{p-1}$.  On the other hand, since $(u, v)$ is a nonnegative solution of \eqref{SPS} in $\mathbb{R}^n \backslash \{0\}$, by Theorem \ref{R-S}, then either $(u, v)$ is trivial, i.e., $u=v\equiv 0$, or $(u, v)$ is of the form
$$
u(x)= k C_0 |x|^{-\frac{2}{p-1}}, ~~ v(x)= l C_0 |x|^{-\frac{2}{p-1}},
$$
where $C_0$ is given by \eqref{SE-02} and $k, l \geq 0 $ (not all  0)  satisfy   \eqref{T011}. We notice that if one of $k$ and $l$ is 0, such as $k=0$, then
$l$ satisfies $\mu_2 l^{2q}=1$.  On the other hand,   if $(u, v)$has the form \eqref{form}, then,  a direct calculation shows that $E(r; u, v)\equiv constant$.

\vskip0.1in

{\it Step 2. } We compute the possible values  of $E(0; u, v)$ and $E(\infty; u, v)$.

\vskip0.1in

By Proposition \ref{P-01} and Lemma \ref{L-02}, we  know that the limit
$$
E(0; u, v)=\lim_{r\to0} E(r; u,v)
$$
exists. For any $\lambda >0$, define the blowing up sequence
$$
u^\lambda(x)=\lambda^{\frac{2}{p-1}} u(\lambda x),~~~v^\lambda(x)=\lambda^{\frac{2}{p-1}} v(\lambda x).
$$
Then $(u^\lambda, v^\lambda)$ is also a nonnegative solution of \eqref{SPS} in $\mathbb{R}^n \backslash \{0\}$. By Lemma \ref{L-01},  $(u^\lambda, v^\lambda)$ is uniformly bounded in each compact subset of $\mathbb{R}^n \backslash \{0\}$. It follows from the interior regularity that $(u^\lambda, v^\lambda)$ is uniformly bounded in $C^{2, \gamma} (K)$ on each compact set $K \subset \mathbb{R}^n \backslash \{0\}$,  for some $0< \gamma <1$.  Hence,  there exists a nonnegative function $(u^0, v^0)\in C^2(\mathbb{R}^n \backslash \{0\})$ and a subsequence $\lambda_j \to 0$ such that $(u^{\lambda_j}, v^{\lambda_j})$ converges to $(u^0, v^0)$ in $C_{loc}^2(\mathbb{R}^n \backslash \{0\})$,  and $(u^0, v^0)$ also satisfies  \eqref{SPS} in $\mathbb{R}^n \backslash \{0\}$.  Moreover, by the scaling invariance of $E$, for any $r>0$,  we have
\begin{equation}\label{S401}
E(r; u^0, v^0)= \lim_{j\to \infty} E(r; u^{\lambda_j}, v^{\lambda_j} ) = \lim_{j\to \infty} E(r \lambda_j;  u, v) =E(0; u, v).
\end{equation}
That is, $E(r; u^0, v^0)$ is a constant for all $r>0$.  By Step 1,  either $(u^0, v^0)$ is trivial, i.e.,  $u^0=v^0\equiv 0$, or $(u^0, v^0)$ is of the form
\begin{equation}\label{S402}
u^0(x)= k C_0 |x|^{-\frac{2}{p-1}}, ~~ v^0(x)= l C_0 |x|^{-\frac{2}{p-1}},
\end{equation}
where $C_0$ is given by \eqref{SE-02} and $k, l \geq 0 $ (not all  0)  satisfy   \eqref{T011}.
If  $(u^0, v^0)$ is trivial, by \eqref{S401}, then $E(0,u,v)=0$. If $(u^0, v^0)$ has the form \eqref{S402}, then a direct  calculation shows that
$$
E(r,u^0,v^0)=-\frac{p-1}{2(p+1)} (k^2 + l^2) C_0^{p+1}
$$
for all $r>0$. Thus , by \eqref{S401}, we get $E(0,u,v)=-A_{k,l}$.

\vskip0.1in
To get the possible value of $E(\infty; u, v)$,  the method is similar to the above argument. We  define the blowing down sequence
$$
u^\lambda(x)=\lambda^{\frac{2}{p-1}} u(\lambda x),~~~v^\lambda(x)=\lambda^{\frac{2}{p-1}} v(\lambda x),
$$
but in this case we will let $\lambda \to +\infty$. As in the above argument,  there exists a nonnegative function $(u^\infty, v^\infty)\in C^2(\mathbb{R}^n \backslash \{0\})$ and a subsequence $\lambda_j \to +\infty$ such that $(u^{\lambda_j}, v^{\lambda_j})$ converges to $(u^\infty, v^\infty)$ in $C_{loc}^2(\mathbb{R}^n \backslash \{0\})$,  and $(u^\infty, v^\infty)$ also satisfies  \eqref{SPS} in $\mathbb{R}^n \backslash \{0\}$.  By the scaling invariance of $E$, for any $r>0$,  we have
\begin{equation}\label{S403}
E(r; u^\infty, v^\infty)= \lim_{j\to \infty} E(r; u^{\lambda_j}, v^{\lambda_j} ) = \lim_{j\to \infty} E(r \lambda_j;  u, v) =E(\infty; u, v).
\end{equation}
The rest of the argument is the same as the proof for $E(0; u, v)$.  We also get
$$
E(\infty;u,v)  \in \left\{ 0, -\frac{p-1}{2(p+1)} (k^2 + l^2) C_0^{p+1} \right\}.
$$

\vskip0.1in

{\it Step 3. } We prove that $E(0; u, v) = 0$ if and only if $(u, v)$ is trivial.

\vskip0.1in

If $E(0;u, v)=0$,  then since $E(r; u, v)$ is nondecreasing for $r>0$ and $E(\infty;u,v)  \in \left\{ 0, -A_{k,l} \right\}$, we must have $E(r; u, v)=0$ for all $r>0$. By Step 1, this implies that either $(u, v)$ is  trivial,   or is of the form \eqref{form} with $k,l\geq 0$ but not all  0.
However, the latter gives $E(0; u, v)=-A_{k,l} <0$, a contradiction. Hence, $(u, v)$ must be trivial. The converse is obvious.

\vskip0.1in

We are now ready to prove Theorem \ref{T02}.  Let $(u, v)$ be a positive solution of \eqref{SPS}.  Obviously, Theorem \ref{T01} yields $E(0; u, v) =-A_{k,l} <0$. We begin with some lemmas.

\begin{lemma}\label{L4-01}
Let $(u, v)$ be a positive solution of \eqref{SPS} with $\frac{n}{n-2} < p < \frac{n+2}{n-2}$. Then
$$
\liminf_{|x| \to 0} |x|^{\frac{2}{p-1}} \left( u(x) + v(x) \right) >0.
$$
\end{lemma}
\begin{proof}
We suppose by contradiction that
$$
\liminf_{|x| \to 0} |x|^{\frac{2}{p-1}} \left( u(x) + v(x) \right) = 0.
$$
Then there exists a sequence of positive numbers $r_i$ converging to 0 such that
\begin{equation}\label{S404}
r_i^{\frac{2}{p-1}} \left( u(r_i) + v(r_i) \right) \to 0 ~~~ \textmd{as} ~   i \to \infty.
\end{equation}
Define
$$
u_i(x)=\frac{u(r_i x)}{u(r_i) + v(r_i)}, ~~~~ v_i(x)=\frac{v(r_i x)}{u(r_i) + v(r_i)}.
$$
By Harnack inequality \eqref{har-01},  $(u_i, v_i)$ is locally uniformly bounded away from the origin. Moreover, $(u_i, v_i)$ satisfies
$$
\begin{cases}
-\Delta u_i =\Big( r_i^{\frac{2}{p-1}} \left( u(r_i) + v(r_i) \right)  \Big)^{p-1}  \left( \mu_1 u_i^{2q+1} + \beta u_i^q v_i^{q+1} \right) \\
-\Delta v_i =\Big( r_i^{\frac{2}{p-1}} \left( u(r_i) + v(r_i) \right)  \Big)^{p-1} \left( \mu_2 v_i^{2q+1} + \beta v_i^q u_i^{q+1} \right)
\end{cases} \textmd{in} ~\mathbb{R}^n \backslash \{0\}.
$$
It follows from the classical gradient estimates that $\nabla u_i$ and $\nabla v_i$ are locally uniformly bounded in $C_{loc}(\mathbb{R}^n \backslash \{0\})$.
Hence, there exists some $C>0$ independent of $i$  such that
$$
|\nabla u (x)|,~ |\nabla v (x)| \leq C r_i^{-1} (u(r_i) + v(r_i))) = o(1) r_i^{-\frac{2}{p-1} - 1}~~~ \textmd{for} ~ \textmd{all} ~ |x|=r_i.
$$
This together with \eqref{S404} easily yield $\lim_{i \to \infty} E(r_i; u, v)=0$. By the monotonicity of $E$, we get $\lim_{r \to 0 } E(r; u, v)=0$. This  contradicts Theorem \ref{T01}.
\end{proof}

\begin{lemma}\label{L4-02}
Let $n=3$ and $(u, v)$ be a nonnegative   solution of \eqref{SPS} with $3< p < 5$.  If
$$
\lim_{|x|\to 0} |x|^{\frac{2}{p-1}} u(x)=0,
$$
then the singularity of $u$ at $x=0$ is removable.
\end{lemma}
\begin{remark}\label{R120}
We remark that there is no additional assumption  for $v$  in Lemma \ref{L4-02}.
\end{remark}
\bp
We consider an auxiliary function $w(x)$ as in \cite{A}, given  by
$$
w(x)=|x|^s u(x),~~~ s>0.
$$
Since
$$
w_{x_i} = s |x|^{s-2} x_i u(x) + |x|^s u_{x_i}
$$
and
$$
w_{x_i x_i}= s(s-2) |x|^{s-4} x_i^2 u + s |x|^{s-2} u +2s |x|^{s-2}x_i u_{x_i} +|x|^s u_{x_i x_i},
$$
we get
$$
\Delta w - \frac{2s}{|x|^2} x \cdot \nabla w = \left[ s(n-2-s) - |x|^2 (\mu_1 u^{2q} + \beta u^{q-1} v^{q+1}) \right]  \frac{w}{|x|^2}.
$$
By our assumption and \eqref{S306} in Lemma \ref{L-01}, we have
$$\lim_{|x| \to 0} |x|^2 (\mu_1 u^{2q} + \beta u^{q-1} v^{q+1} ) =0. $$
Hence, for any $0 < s <1$ there exists $R_s>0$ such that
$$
\Delta w - \frac{2s}{|x|^2} x \cdot \nabla w \geq 0~~~~~ \textmd{in}  ~ B_{R_s} \backslash \{0\}.
$$
On the other hand, since $u(x) \leq C |x|^{-\frac{2}{p-1}}$ and $\frac{2}{p-1} < n-2$, we deduce that there exists $\epsilon>0$ independent of $s$ such that $w(x)=O(|x|^{2-n+\epsilon})$.  By Theorem 1 and Remark 1 in \cite{G-S},  we obtain
$$
u(x) \leq \left( \frac{R_s}{|x|} \right)^s \max_{\partial B_{R_s}} u(x),~~~~~\textmd{for}  ~ x\in B_{R_s} \backslash \{0\}.
$$
Therefore, we have  $u \in L^\gamma(B_1)$  for all $\gamma>1$.

\vskip0.1in

Now we write  the first equation of \eqref{SPS} in the form $ -\Delta u = f(x) u$,
where  $f=\mu_1 u^{2q} + \beta u^{q-1} v^{q+1}$.  Since $v^{q+1}\leq C |x|^{-\frac{p+1}{p-1}}$ and $p>3$, we  get
$$
v^{q+1} \in L^{\frac{3}{2-\theta}} (B_1)
$$
for some $\theta>0$  small.  Hence $f\in L^{\frac{3}{2-\theta_0}} (B_1)$ for some $\theta_0>0$ smaller.   By Theorem 11 of Serrin \cite{S} , we obtain that  0 is a removable singularity of $u$, i.e., $u(x)$ can be extended to a continuous $H^1$ weak solution in the entire ball $B_1$.
\ep

\begin{lemma}\label{L4-03}
Let $n\geq 4$ and $(w_1, w_2)$ be a nonnegative   solution of \eqref{S301} with $\frac{n}{n-2}< p < \frac{n+2}{n-2}$.  If
$$
\liminf_{t \to +\infty} w_1(t)=\liminf_{t \to +\infty} w_2 (t)= 0,
$$
then
$$
\liminf_{t \to +\infty} (w_1 + w_2) (t)=0.
$$
\end{lemma}
\bp
If there exists $T$ such that $w^\prime_1 (t) \leq 0$ for all $t > T$, then we have
$$
\lim_{t \to +\infty} w_1 (t)= 0,
$$
and hence $\liminf_{t \to +\infty} (w_1 + w_2) (t)=0$. Otherwise, there exits a sequence of local minimum point $t_i$ of $w_1$ such that $t_i\to +\infty$ and  $w_1(t_i) \to 0$.  By $w^{\prime \prime}_1 (t_i) \geq 0$ and the first equation of \eqref{S301},  we have
$$
\beta w_2^{q+1}(t_i) \leq \sigma w_1^{1-q} (t_i).
$$
By $n\geq 4$, we have $0< q < 1$, and so $w_2(t_i) \to 0$. Hence $(w_1 + w_2) (t_i)  \to 0$.  The desired conclusion follows.
\ep

\vskip0.1in

\noindent {\bf Proof of (1) in Theorem \ref{T02}.}  We first prove that  \eqref{SPS} has no positive semi-singular  solutions at 0 for  $n \geq 4$.

\vskip0.1in
Let $n\geq 4$. Suppose by contradiction that $(u, v)$ is a positive semi-singular  solution at 0 of  \eqref{SPS}. Without loss of generality, we assume
\begin{equation}\label{PT00}
\lim_{r\to 0} u(r) =+\infty ~~~~ \textmd{and} ~~~~ \lim_{r\to 0} v(r) < +\infty.
\end{equation}
Then there exists $C_0 > 0$ such that
\begin{equation}\label{PT01}
u(r) \geq C_0 r^{-\frac{2}{p-1}},~~~~~~ \textmd{for}  ~ r\in (0, 1].
\end{equation}
Indeed, if $\eqref{PT01}$ does not hold,  then there exists $r_i \to 0$ such that $ r_n^{\frac{2}{p-1}} u(r_i) \to 0$.  Since $v$ is bounded near the origin, hence
$$
\liminf_{|x| \to 0} |x|^{\frac{2}{p-1}} (u(x) + v(x))=0.
$$
This contradicts  Lemma \ref{L4-01}.

\vskip0.1in

On the other hand,  by the decreasing property of $v$, there exists $C_1 >0$ such that $v(r) > C_1$ for all $(0,1]$.  By Lemma \ref{L-01}, we  have $r^{n-1} v^\prime (r) \to 0$ as $r\to 0$.  Since $v(x)$ is radially  symmetry, we write the second equation of \eqref{SPS} in the form
$$
-(r^{n-1} v^\prime (r))^\prime = r^{n-1} (\mu_2 v^{p+1} + \beta v^q u^{q+1}).
$$
Integrating  from 0 to $r$, we get
$$
\aligned
- r^{n-1} v^\prime (r) & = \int_0^r \rho^{n-1} (\mu_2 v^{p+1} + \beta v^q u^{q+1}) d\rho \geq \beta \int_0^r \rho^{n-1} v^q u^{q+1}  d\rho \\
& \geq C \int_0^r \rho^{n-1} \rho^{-\frac{p+1}{p-1}} d\rho   = C r^{n-\frac{p+1}{p-1}},
\endaligned
$$
that is, we have $-v^\prime (r) \geq C r^{-\frac{2}{p-1}}$ for all $r \in (0, 1]$. For any $r\in (0, 1)$, integrating  from $r$ to $1$, we obtain
$$
v(r) -v(1) \geq  C \int_r^1 \rho^{-\frac{2}{p-1}}d\rho  = C \left(\frac{2}{p-1} - 1 \right)^{-1} (r^{1-\frac{2}{p-1}}  -1).
$$
We note that $1-\frac{2}{p-1} <0$ if $\frac{n}{n-2} < p <  \frac{n+2}{n-2}$ and $n \geq 4$.  Hence we get $v(r) \to +\infty$ as $r \to 0$, a contradiction with
\eqref{PT00}. This finishes the proof of  (1) in Theorem  \ref{T02}.

\vskip0.2in

\noindent{\bf Proof of (2) in Theorem  \ref{T02}.} Let $(u, v)$ be a positive solution of \eqref{SPS}. We will prove that \eqref{T021} holds under the assumption of (2) in Theorem \ref{T02}.  By Lemma \ref{L-01}, we only need to prove
\begin{equation}\label{PT02}
\liminf_{|x|\to 0} |x|^{\frac{2}{p-1}} u(x) > 0 ~~~~ \textmd{and} ~~~~\liminf_{|x|\to 0} |x|^{\frac{2}{p-1}} v(x) > 0.
\end{equation}
This is equivalent to prove
\begin{equation}\label{PT03}
\liminf_{t\to +\infty} w_1(t) > 0 ~~~~ \textmd{and} ~~~~\liminf_{t \to +\infty} w_2(t) > 0.
\end{equation}

\vskip0.1in

We prove \eqref{PT03} by considering  two cases separately.

\vskip0.1in

{\bf Case 1.}  $n\geq 4$.

\vskip0.1in

Suppose that \eqref{PT03} does not hold. Then by Lemma \ref{L4-01} and \ref{L4-03}, we may assume, without loss of generality, that
\begin{equation}\label{PT04}
\liminf_{t\to +\infty} w_1(t) = 0 ~~~~ \textmd{and} ~~~~\liminf_{t \to +\infty} w_2(t) =C_1> 0.
\end{equation}
If $\limsup_{t\to +\infty} w_1(t) > 0$, then there exits a sequence of local minimum point $t_i$ of $w_1$ such that $t_i\to +\infty$ and  $w_1(t_i) \to 0$.  By $w^{\prime \prime}_1 (t_i) \geq 0$ and the first equation of \eqref{S301},  we have
$$
\beta w_2^{q+1}(t_i) \leq \sigma w_1^{1-q} (t_i).
$$
By $n\geq 4$, we have $0< q < 1$, and so $w_2(t_i) \to 0$, a contradiction with \eqref{PT04}. Hence,  we get
\begin{equation}\label{PT05}
\lim_{t\to +\infty} w_1(t) = 0.
\end{equation}
It follows from the first equation of \eqref{S301}, Corollary \ref{C-01} and \eqref{PT04} that
\begin{equation}\label{PT06}
\aligned
w_1^{\prime\prime}(t) &= -\tau w_1^{\prime}(t)  + \sigma w_1(t) - ( \mu_1 w_1^{2q+1} + \beta w_1^q w_2^{q+1})(t) \\
& \leq \tau\delta w_1(t) + \sigma w_1(t) - ( \mu_1 w_1^{2q+1} + \beta w_1^q w_2^{q+1})(t)  \\
& = \delta^2 w_1(t) - ( \mu_1 w_1^{2q+1} + \beta w_1^q w_2^{q+1})(t) \\
& = w_1^q (t) ( \delta^2 w_1^{1-q} -\mu_1 w_1^{q+1} - \beta w_2^{q+1} )(t) < 0~~~~~ \textmd{for} ~ \textmd{all} ~~ t  \geq T_0
\endaligned
\end{equation}
with some $T_0 >0$ large.
If $w_1^\prime (t) >0$ for all $t>T_0$, then $w_1(t) \geq w_1(T_0) >0 $ for all $t>T_0$, a contradiction with \eqref{PT05}. So there exists $T_1 > T_0$  such that
$w_1^\prime (T_1) \leq 0$. By \eqref{PT06}, we have $w_1^\prime (t) \leq 0$ for all $t \geq T_1$. By Corollary \ref{C-01},
$$
|w_1^\prime(t)| = -w_1^\prime (t) \leq \delta w_1(t)~~~~ \textmd{for} ~ \textmd{all} ~~ t  \geq T_1.
$$
So we get $|w_1^\prime(t)|  \to 0$ as $t \to +\infty$ by \eqref{PT05}.  We easily deduce from \eqref{PT06} that $w_1^\prime (t) >0$ for all $t  \geq T_0$, a contradiction with \eqref{PT05}.  Therefore \eqref{PT03} holds.

\vskip0.1in

{\bf Case 2.}  $n= 3$ and assume that $\eqref{SPS}$ has no positive semi-singular  solutions at 0.

\vskip0.1in

Under this assumption, we  have
\begin{equation}\label{PT07}
\lim_{r \to 0} u(r) = \lim_{r \to 0} v(r)  = +\infty.
\end{equation}
Suppose by contradiction that \eqref{PT03} does not hold. Without loss of generality, we assume that $\liminf_{t \to +\infty} w_1(t)=0$.  Note  that $1 < q <2$ by $n=3$. By Lemma \ref{L-01},  we know that $w_1$ and $w_2$ are uniformly bounded for $t\in \mathbb{R}$.  Then there exists $c>0$ such that
\begin{equation}\label{PT08}
w_1^{\prime\prime} (t) +\tau w_1^\prime (t)  \geq \sigma w_1(t) -c w_1^q(t)~~~~~~\textmd{in} ~ \mathbb{R}.
\end{equation}
If $\limsup_{t \to +\infty} w_1(t) >0$, then there exits a sequence of local minimum point $t_i$ of $w_1$ such that $t_i\to +\infty$ and  $w_1(t_i) \to 0$.
By \eqref{PT08}, there exists $\epsilon >0$ small such that
\begin{equation}\label{PT09}
\frac{d}{dt} \left( e^{\tau t} w_1^{\prime} (t) \right) = e^{\tau t} ( w_1^{\prime\prime} (t) +\tau w_1^\prime (t))   >0
\end{equation}
whenever $w_1(t) \leq 2\epsilon$.  Hence, there exist $t_i^* < t_i$ such that $w_1(t_i^*)=\epsilon$ and $w^\prime_1(t) < 0$ for  $t\in [t_i^*, t_i)$. So we have (making $\epsilon$ smaller if necessary)
\begin{equation}\label{PT10}
w_1^{\prime\prime} (t)   \geq \sigma w_1(t) -c w_1^q(t) \geq \frac{\sigma}{2}w_1(t)~~~~~~\textmd{in} ~  [t_i^*, t_i).
\end{equation}
Hence $(w_1^\prime)^2 - \frac{\sigma}{2} w_1^2$ is nonincreasing in $[t_i^*, t_i)$. In particular, we have
$$
(w_1^\prime)^2(t) - \frac{\sigma}{2} w_1^2(t) \geq - \frac{\sigma}{2} w_1^2(t_i)
$$
for $t \in [t_i^*, t_i)$. Integrating the inequality above, we get
$$
t_i - t_i^* \leq \sqrt{\frac{2}{\sigma}}   \int_{w_1(t_i)}^{w_1(t_i^*)} \frac{dw}{\sqrt{w^2 - w_1^2(t_i)}} \leq \sqrt{\frac{2}{\sigma}}  \log  \frac{2 w_1(t_i^*)}{w_1(t_i)}.
$$
Therefore,
$$
w_1(t_i) \leq 2 \epsilon e^{-\sqrt{\frac{\sigma}{2}} (t_i - t_i^*)}.
$$
Let
$$
W_1(t):= 2 \epsilon e^{-\sqrt{\frac{\sigma}{2}} (t - t_i^*)}.
$$
Then we have
\begin{equation}\label{PT11}
\begin{cases}
-w_1^{\prime\prime} (t) + \frac{\sigma}{2} w_1(t) \leq 0 = -W_1^{\prime\prime} (t) + \frac{\sigma}{2} W_1(t)~~~~~~~\textmd{in} ~  [t_i^*, t_i),\\
w_1(t_i^*) \leq  W_1(t_i^*), \\
w_1(t_i) \leq W_1(t_i).
\end{cases}
\end{equation}
The maximum principle gives
\begin{equation}\label{PT12}
w_1(t) \leq W_1(t)= 2\epsilon e^{-\sqrt{\frac{\sigma}{2}} (t - t_i^*)} ~~~~~~~\textmd{in} ~  [t_i^*, t_i).
\end{equation}
By Lemma \ref{L-01},  there exists $C>0$ such that $w_1, w_2, |w^\prime_1|, |w^\prime_2| \leq C$ for all $t\in \mathbb{R}$.  Hence, up to a subsequence, $w_1(\cdot + t_i^*) \to \widetilde{w}_1 \geq 0$ and $w_2(\cdot + t_i^*) \to \widetilde{w}_2 \geq 0$ uniformly in $C_{loc}^2(\mathbb{R})$, where $\widetilde{w}_1(0)=1$ and $(\widetilde{w}_1, \widetilde{w}_2)$ satisfies
$$
\begin{cases}
\widetilde{w}_1^{\prime\prime} + \tau \widetilde{w}_1^{\prime} -  \sigma \widetilde{w}_1 + \mu_1 \widetilde{w}_1^{2q+1} + \beta \widetilde{w}_1^q \widetilde{w}_2^{q+1}=0 \\
\widetilde{w}_2^{\prime\prime} + \tau \widetilde{w}_2^{\prime} -  \sigma \widetilde{w}_2 + \mu_2 \widetilde{w}_2^{2q+1} + \beta \widetilde{w}_2^q \widetilde{w}_1^{q+1}=0
\end{cases} t\in \mathbb{R}.
$$
Furthermore,
 $$
\aligned
&\frac{1}{2} \left( |\widetilde{w}_1^{\prime}|^2  + |\widetilde{w}_2^{\prime}|^2 - \sigma \left( \widetilde{w}_1^2 + \widetilde{w}_2^2 \right) \right) (t)\\
& ~ +\frac{1}{p+1} \left( \mu_1 \widetilde{w}_1^{p+1} + 2\beta \widetilde{w}_1^{q+1} \widetilde{w}_2^{q+1} + \mu_2 \widetilde{w}_2^{p+1} \right) (t) \equiv \Psi(+\infty).
\endaligned
$$
Here $\Psi$ is defined in \eqref{S303}.  Clearly, the limit $\Psi(+\infty):=\lim_{t \to +\infty} \Psi(t)$ exists. Therefore, we have
$$
 - \tau \left[ (\widetilde{w}_1^{\prime})^2 + (\widetilde{w}_2^{\prime})^2 \right](t) \equiv 0 ~~~~~ t\in \mathbb{R}.
$$
So $\widetilde{w}_1(t) \equiv 1$ for $t \in \mathbb{R}$.  On the other hand, since $w_1^\prime (t) \leq 0$ in $[t_i^*, t_i]$,  by Corollary \ref{C-01}, we obtain $|w_1^\prime (t)|=-w_1^\prime (t) \leq \delta w_1(t)$ for all  $t\in [t_i^*, t_i]$.  By the mean value theorem, we have
$$
t_i - t_i^* \geq \frac{1}{\delta} \ln\frac{\epsilon}{w_1(t_i)} \to +\infty.
$$
This together with \eqref{PT12} yield
$$
\widetilde{w}_1(t) \leq 2 \epsilon e^{-\sqrt{\frac{\sigma}{2}} t},~~~~~ \textmd{for} ~ t >0.
$$
This is a contradiction with $\widetilde{w}_1(t) \equiv 1$ for $t \in \mathbb{R}$.  Therefore, $\lim_{t \to +\infty} w_1(t)=0$, namely, $\lim_{r\to 0} r^{\delta} u(r)=0$. By Lemma \ref{L4-02}, the singularity of $u$ at 0 is removable,  and so $\lim_{r\to 0} u(r) < +\infty$, a contradiction with \eqref{PT07}.  This finishes the proof of  (2) in Theorem  \ref{T02}.

\vskip0.2in

Next we are going to prove  (3) and (4) in Theorem  \ref{T02}. To this end,  we define  the Kelvin transform
$$
\bar{u}(x)=\frac{1}{|x|^{n-2}} u\left( \frac{x}{|x|^2}\right),~~~~~~ \bar{v}(x)=\frac{1}{|x|^{n-2}} v \left( \frac{x}{|x|^2}\right).
$$
Then $(\bar{u}, \bar{v})$ satisfies
\begin{equation}\label{SPS-01}
\begin{cases}
-\Delta \bar{u} =|x|^\alpha (\mu_1 \bar{u}^{2q+1} + \beta \bar{u}^q \bar{v}^{q+1} ) \\
-\Delta \bar{v} =|x|^\alpha (\mu_2 \bar{v}^{2q+1} + \beta \bar{v}^q \bar{u}^{q+1} )
\end{cases} \textmd{in} ~\mathbb{R}^n \backslash \{0\},
\end{equation}
where $\alpha:=p(n-2) - (n+2) \in (-2, 0)$.  By Lemma \ref{L-01}, we also have
\begin{equation}\label{Ste301}
\bar{u}(x), ~ \bar{v}(x) \leq C |x|^{-\left( n-2-\frac{2}{p-1} \right)} = C |x|^{-\frac{2+\alpha}{p-1}} ~~~~~~\textmd{in} ~\mathbb{R}^n \backslash \{0\}
\end{equation}
and
\begin{equation}\label{Ste302}
|\nabla \bar{u}(x)|, ~ |\nabla \bar{v}(x)|  \leq C |x|^{-\left( n-1-\frac{2}{p-1} \right)} = C |x|^{-\frac{2+\alpha}{p-1} -1} ~~~~~~\textmd{in} ~\mathbb{R}^n \backslash \{0\}.
\end{equation}
Denote $\delta_0=\frac{2+\alpha}{p-1}$. Let $t=-\ln r$ and
$$
\bar{w}_1(t):=r^{\delta_0} \bar{u}(r)= e^{-\delta_0  t} \bar{u}(e^{-t}),~~~~~~~~ \bar{w}_2(t):=r^{\delta_0} \bar{v}(r)= e^{-\delta_0 t} \bar{v}(e^{-t}).
$$
Then by a direct  calculation  $(\bar{w}_1, \bar{w}_2)$ satisfies
\begin{equation}\label{Ste303}
\begin{cases}
\bar{w}_1^{\prime\prime} + \tau_0 \bar{w}_1^{\prime} -  \sigma_0 \bar{w}_1 + \mu_1 \bar{w}_1^{2q+1} + \beta \bar{w}_1^q \bar{w}_2^{q+1}=0 \\
\bar{w}_2^{\prime\prime} + \tau_0 \bar{w}_2^{\prime} -  \sigma_0 \bar{w}_2 + \mu_2 \bar{w}_2^{2q+1} + \beta \bar{w}_2^q \bar{w}_1^{q+1}=0
\end{cases} t\in \mathbb{R},
\end{equation}
where
$$
\tau_0=\frac{n-2}{p-1} \left(\frac{n+2+2\alpha}{n-2} -p \right) <  0,~~~\sigma_0=\frac{(2+\alpha)(n-2)}{(p-1)^2} \left(p- \frac{n+\alpha}{n-2} \right) >0.
$$
We note that there are some differences between  system \eqref{Ste303} and system \eqref{S301}.  In particular, the coefficient $\tau_0$ in  \eqref{Ste303} is {\it less than} 0, and the coefficient $\tau$ in \eqref{S301} is {\it greater than} 0. By Lemma \ref{R-S}, both $\bar{u}$ and  $\bar{v}$ are also radially symmetric. Similar to the proof of Lemma \ref{L-03} and Corollary \ref{C-01}, we  have
\begin{lemma}\label{S3L01}
Assume that $\frac{n}{n-2} < p < \frac{n+2}{n-2}$ and $\alpha=p(n-2) - (n+2)$.

\begin{itemize}
\item [(1)] Let $(\bar{u}, \bar{v})$ be a positive solution of \eqref{SPS-01}. Then both $\bar{u}^{\prime}(r) < 0$ and $\bar{v}^{\prime}(r) < 0$ for all $r>0$.

\item [(2)] Let $(\bar{w}_1, \bar{w}_2)$ be a positive solution of \eqref{Ste303}. Then $\bar{w}_i^\prime (t) > -\delta_0 \bar{w}_i(t)$ for all $t\in \mathbb{R}$ and $i=1, 2$.
\end{itemize}
\end{lemma}
We define
\begin{equation}\label{S3E03}
\aligned
\bar{\Psi}  (t; \bar{w}_1, \bar{w}_2):= & \frac{1}{2} \left( |\bar{w}_1^{\prime}|^2  + |\bar{w}_2^{\prime}|^2 - \sigma_0 \left( \bar{w}_1^2 + \bar{w}_2^2 \right) \right) (t) \\
 &~ +\frac{1}{p+1} \left( \mu_1 \bar{w}_1^{p+1} + 2\beta \bar{w}_1^{q+1} \bar{w}_2^{q+1} + \mu_2 \bar{w}_2^{p+1} \right) (t).
\endaligned
\end{equation}
Multiplying the first equation of \eqref{Ste303} by $\bar{w}_1^{\prime}$, the second equation of \eqref{Ste303} by $\bar{w}_2^{\prime}$, we easily obtain
\begin{lemma}\label{S3L02}
Let $(\bar{w}_1, \bar{w}_2)$ be a nonnegative  solution of \eqref{Ste303} with $\frac{n}{n-2} < p < \frac{n+2}{n-2}$. Then $\bar{\Psi}$ is nondecreasing  and uniformly bounded for $t\in \mathbb{R}$. Moreover,
\begin{equation}\label{S3E030}
\frac{d}{dt}\bar{\Psi} (t; \bar{w}_1, \bar{w}_2)=  - \tau_0 \left[ (\bar{w}_1^{\prime})^2 + (\bar{w}_2^{\prime})^2 \right](t).
\end{equation}
\end{lemma}
Hence the limit $\bar{\Psi} (+\infty; \bar{w}_1, \bar{w}_2)=\lim_{t \to +\infty} \bar{\Psi} (t; \bar{w}_1, \bar{w}_2)$ exists. Further, we also have
\begin{lemma}\label{S3L03}
$\bar{\Psi} (+\infty; \bar{w}_1, \bar{w}_2) \in \left\{0, -\frac{p-1}{2(p+1)} (k^2 + l^2) \sigma_0^{\frac{p+1}{p-1}} \right\}$,
where $k, l \geq 0 $ (not all  0)  satisfy  \eqref{T011}.
\end{lemma}
\bp
Given any sequence $t_i \to +\infty$, up to a subsequence, $\bar{w}_1(\cdot + t_i) \to z_1\geq 0$ and $\bar{w}_2(\cdot + t_i) \to z_2\geq 0$ in $C_{loc}^2(\mathbb{R})$.  Then $(z_1, z_2)$  satisfies \eqref{Ste303} and
\begin{equation}\label{S3E04}
\bar{\Psi} (t; z_1, z_2)=\lim_{i\to\infty}\bar{\Psi} (t; \bar{w}_1(\cdot + t_i), \bar{w}_2(\cdot + t_i))
\ee
\be
=\lim_{i\to\infty}\bar{\Psi} (t+t_i; \bar{w}_1, \bar{w}_2)=\bar{\Psi} (+\infty; \bar{w}_1, \bar{w}_2).
\end{equation}
That is, $\bar{\Psi} (t; z_1, z_2)$ is a constant for all $t\in \mathbb{R}$. By \eqref{S3E03}, both $z_1$ and $z_2$ are also constant. Then, either $z_1=z_2\equiv 0$, or $z_1=k \sigma_0^{\frac{1}{p-1}}$ and $z_2=l \sigma_0^{\frac{1}{p-1}}$, where $k, l \geq 0 $ (not all 0)  satisfy  \eqref{T011}.  The conclusion follows easily by \eqref{S3E04}.
\ep

\vskip0.1in

\noindent {\bf Proof of (3) in Theorem \ref{T02}.}
We prove the nonexistence of positive semi-singular solutions at $\infty$ for $n\geq 4$.
\vskip0.1in

Let $n\geq 4$. Suppose by contradiction that $(u, v)$ is a positive semi-singular  solution of  \eqref{SPS} at $\infty$.  Without loss of generality, we assume
\begin{equation}\label{3TPf01}
u(r) \leq C r^{-(n-2)} ~~~~ \textmd{and} ~~~~  v(r) \geq C r^{-\frac{2}{p-1}}~~~~~~~ \textmd{for} ~~ r ~\textmd{large}.
\end{equation}
Then
\begin{equation}\label{3TPf02}
\bar{u}(r) \leq C ~~~~ \textmd{and} ~~~~  \bar{v}(r) \geq C r^{-(n-2-\frac{2}{p-1})}~~~~~~~ \textmd{for} ~~ 0 < r < r_0
\end{equation}
with some $ 0 < r_0 < 1$.
Clearly there exists $C_0>0$ such that $\bar{u} (r) \geq C_0$ for all $r\in (0, r_0]$. By \eqref{Ste302}, we have $r^{n-1}\bar{u}^\prime(r) \to 0$ as $r \to 0$. Since
$$
-(r^{n-1} \bar{u}^\prime (r))^\prime = r^{n-1} r^\alpha  (\mu_1 \bar{u}^{p+1} + \beta \bar{u}^q \bar{v}^{q+1}).
$$
Integrating  from 0 to $r$, we get
$$
\aligned
- r^{n-1} \bar{u}^\prime (r) & = \int_0^r \rho^{n-1} \rho^\alpha ( \mu_1 \bar{u}^{p+1} + \beta \bar{u}^q \bar{v}^{q+1} ) d\rho \geq \beta \int_0^r \rho^{n-1+\alpha} \bar{u}^q \bar{v}^{q+1}  d\rho \\
& \geq C \int_0^r \rho^{n-1+\alpha} \rho^{-(n-2-\frac{2}{p-1}) \frac{p+1}{2}} d\rho   = C r^{\frac{p-1}{2}(n-2) + \frac{2}{p-1} -1 },
\endaligned
$$
namely $-\bar{u}^\prime (r) \geq C r^{\frac{p-1}{2}(n-2) + \frac{2}{p-1} -n}$ for all $ r\in (0,r_0]$. If $n\geq 4$, then a simple analysis gives $\frac{p-1}{2}(n-2) + \frac{2}{p-1} -n  < -1$ for all  $\frac{n}{n-2} < p < \frac{n+2}{n-2}$. This implies for any $r\in (0, r_0)$ that
$$
\bar{u}(r) - \bar{u}(r_0) \geq C \int_r^{r_0} \rho^{\frac{p-1}{2}(n-2) + \frac{2}{p-1} -n} d\rho \geq C \int_r^{r_0} \rho^{-1} d\rho = C(-\ln r + \ln r_0).
$$
Hence, $\bar{u}(r) \to +\infty$ as $r \to 0$, a contradiction with \eqref{3TPf02}. This completes the proof of (3) in Theorem \ref{T02}.

\vskip0.1in

\noindent {\bf Proof of (4) in Theorem \ref{T02}.} We now prove  (4) in Theorem \ref{T02} by discussing two steps separately.

\vskip0.1in
{\it Step 1.} $\bar{\Psi} (+\infty; \bar{w}_1, \bar{w}_2)=0$.

\vskip0.1in

In this case, by the proof of Lemma \ref{S3L03},  we know that
\begin{equation}\label{S3E05}
\lim_{t \to +\infty} \bar{w}_1(t) = \lim_{t \to +\infty} \bar{w}_2(t)=0.
\end{equation}
Let $\bar{w}=\bar{w}_1 + \bar{w}_2$. It follows from \eqref{Ste303} and Lemma \ref{S3L01} (2) that
$$
\aligned
\bar{w}^{\prime\prime}  & \geq -\tau_0 \bar{w}^{\prime} + \sigma_0 \bar{w} - \bar{C} \bar{w}^p \\
& \geq (\tau_0\delta_0 + \sigma_0) \bar{w} -\bar{C} \bar{w}^p = \delta_0^2 \bar{w} -  \bar{C} \bar{w}^p,
\endaligned
$$
where $\bar{C}= \bar{C}(\mu_1, \mu_2, \beta, p)$. By \eqref{S3E05}, $\bar{w}^{\prime\prime}(t) > 0$ for $t> T_1$ with some $T_1 >0$ large, and hence
\begin{equation}\label{S3E06}
\bar{w}^{\prime} (t) <  0 ~~~~ \textmd{for}~ t > T_1.
\end{equation}
Otherwise,  there exists $t_1 > T_1$ such that $\bar{w}^{\prime}(t_1) >   0$, and then $\bar{w}^{\prime}(t) \geq \bar{w}^{\prime}(t_1) >0$ for all $t> t_1$. We get $\bar{w}(t) \to +\infty$ as $t \to +\infty$, a contradiction with \eqref{S3E05}.  Let $\bar{z}(t)=\bar{w}^{\prime}(t) + \delta_0\bar{w}(t)$. Then
\begin{equation}\label{S3E07}
\bar{z}^\prime - \delta_0 \bar{z} = \bar{w}^{\prime\prime} - \delta_0^2 \bar{w} \geq  -\bar{C} \bar{w}^p.
\end{equation}
By \eqref{Ste301} and \eqref{Ste302}, we easily deduce that $\bar{z}$ is bounded. It follows from \eqref{S3E06} and \eqref{S3E07} that
$$
\bar{z}(t) \leq \bar{C} e^{\delta_0 t} \int_t^{+\infty}  e^{- \delta_0 \rho} \bar{w}^p(\rho) d\rho \leq \frac{\bar{C}}{\delta_0} w^p(t),~~~~ \forall ~ t> T_1.
$$
That is,
$$
\bar{w}^{\prime}(t) + \delta_0\bar{w}(t)  \leq \frac{\bar{C}}{\delta_0} w^p(t),~~~~ \forall ~ t> T_1.
$$
From this we easily obtain
\begin{equation}\label{S3E08}
\frac{d}{dt} \left[(e^{\delta_0 t} \bar{w}(t))^{1-p} - \frac{\bar{C}}{\delta_0^2} e^{(1-p)\delta_0 t} \right] \geq 0,~~~~ \forall ~ t> T_1.
\end{equation}
If $\limsup_{t \to +\infty} e^{\delta_0 t}\bar{w}(t) =+\infty$, then there exists a sequence $t_i \to +\infty$ such that $(e^{\delta_0 t_i} \bar{w}(t_i))^{1-p} - \frac{\bar{C}}{\delta_0^2} e^{(1-p)\delta_0 t_i} \to 0$. By \eqref{S3E08}, we have
$$
(e^{\delta_0 t} \bar{w}(t))^{1-p} - \frac{\bar{C}}{\delta_0^2} e^{(1-p)\delta_0 t}  \leq 0,~~~~ \forall ~ t> T_1.
$$
This implies that $\bar{w}(t) \geq \left(\frac{\delta_0}{\bar{C}}\right)^{p-1}$ for all $t>T_1$, a contradiction with \eqref{S3E05}. Therefore, we have $\limsup_{t \to +\infty} e^{\delta_0 t}\bar{w}(t) < +\infty$ and so $e^{\delta_0 t}\bar{w}(t)  \leq C$ uniformly  for $t>0$ large enough. We obtain that $\bar{u} + \bar{v} \leq C$ uniformly  for $r>0$ small. That is,
$$
u(x) + v(x) \leq C |x|^{-(n-2)}
$$
uniformly for $|x|$ large.

\vskip0.1in

{\it Step 2.} $\bar{\Psi} (+\infty; \bar{w}_1, \bar{w}_2) < 0$.

\vskip0.1in

Remark that, in this case, the limits $\lim_{t \to +\infty} \bar{w}_1(t)$ and  $\lim_{t \to +\infty} \bar{w}_2(t)$ {\it cannot be guaranteed to exist}  by the proof of Lemma \ref{S3L03}. See Remark \ref{R-02}.

\vskip0.1in

 We claim

\begin{equation}\label{3PrT01}
\liminf_{t\to +\infty} \bar{w}_1(t) > 0 ~~~~ \textmd{and} ~~~~\liminf_{t \to +\infty} \bar{w}_2(t) > 0.
\end{equation}
\vskip0.1in

We prove this claim  by considering two cases separately.

\vskip0.1in

{\bf Case 1.} $n \geq 4$.

\vskip0.1in

We just need to modify the  proof of Case 1 in  Theorem \ref{T02} (2), the main difference is $\tau_0 <0$ here and $\tau >0$ there.  Suppose that \eqref{3PrT01} does not hold. If
\begin{equation}\label{3PrT02}
\liminf_{t\to +\infty} \bar{w}_1(t) = 0 ~~~~ \textmd{and} ~~~~\liminf_{t \to +\infty} \bar{w}_2(t) = 0.
\end{equation}
then the same argument as that of Lemma \ref{L4-03} gives $\liminf_{t\to +\infty} (\bar{w}_1 + \bar{w}_2) (t) = 0$. We take $t_i \to +\infty$ such that $(\bar{w}_1 + \bar{w}_2) (t_i) \to 0 $. Then up to a subsequence, $\bar{w}_1(\cdot +t_i) \to \bar{z}_1$ and $\bar{w}_2(\cdot +t_i) \to \bar{z}_2$ uniformly in $C_{loc}^2(\mathbb{R})$, where $\bar{z}_1(0)= \bar{z}_2(0)=0$ and $(\bar{z}_1, \bar{z}_2)$ satisfies \eqref{Ste303}.  Moreover,
$\bar{\Psi}  (t; \bar{z}_1, \bar{z}_2) \equiv \bar{\Psi} (+\infty; \bar{w}_1, \bar{w}_2) $ for all $ t \in \mathbb{R}$. By \eqref{S3E03}, $\bar{z}_1=\bar{z}_2\equiv 0$,  a contradiction with $\bar{\Psi} (+\infty; \bar{w}_1, \bar{w}_2) <0$. Hence \eqref{3PrT02}  is impossible.  Without loss of generality, we may assume that
\begin{equation}\label{3PrT03}
\liminf_{t\to +\infty} \bar{w}_1(t) = 0 ~~~~ \textmd{and} ~~~~\liminf_{t \to +\infty} \bar{w}_2(t) =C_2> 0.
\end{equation}
The same argument as that of Case 1 in  Theorem \ref{T02} (2) yields $\lim_{t\to +\infty} \bar{w}_1(t) = 0 $. It follows from the first equation of \eqref{Ste303} that
\begin{equation}\label{3PrT04}
\bar{w}_1^{\prime\prime} + \tau_0 \bar{w}_1^{\prime} =\bar{w}_1^q (\sigma_0 \bar{w}_1^{1-q} - \mu_1 \bar{w}_1^{q+1} - \beta \bar{w}_2^{q+1}) <0 ~~~~~ \textmd{for} ~ \textmd{all} ~~ t  \geq T_3
\end{equation}
with some $T_3 >0$ large, and then $e^{\tau_0 t}\bar{w}_1^{\prime}(t)$ is strictly decreasing for $t > T_3$.  If there exists a sequence $t_i\to +\infty$ such that $\bar{w}_1^{\prime} (t_i) \geq 0$, then $\bar{w}_1^{\prime} (t) \geq 0$ for all $t > T_3$, a contradiction with $\lim_{t\to +\infty} \bar{w}_1(t) = 0$. Hence there exists $T_4 > T_3$ such that $\bar{w}_1^{\prime} (t) < 0$ for all $t > T_4$. By \eqref{3PrT04} and $\tau_0 <0$,  we have $\bar{w}_1^{\prime\prime}(t) < 0$ for $t > T_4$. On the other hand, by Lemma \ref{S3L01} (2), $|\bar{w}_1^{\prime}(t)|=-\bar{w}_1^{\prime}(t) \leq \delta_0 \bar{w}_1(t) \to 0$ as $t\to +\infty$.  We get $\bar{w}_1^{\prime}(t) >0$ for $t > T_4$, a contradiction with $\lim_{t\to +\infty} \bar{w}_1(t) = 0$. Therefore, \eqref{3PrT01} holds.

\vskip0.1in

{\bf Case 2.} $n = 3$ and assume that the system \eqref{SPS} has no positive semi-singular  solutions at $\infty$. 

\vskip0.1in

Suppose by contradiction that \eqref{3PrT01} does not hold, without loss of generality, we assume $\liminf_{t \to +\infty} \bar{w}_1(t)=0$. Since $\bar{w}_1(t)$ and $\bar{w}_2(t)$ are uniformly bounded for $t\in \mathbb{R}$, by Lemma \ref{S3L01}, there exists $\bar{c} >0$ such that
\begin{equation}\label{3PST01}
\bar{w}_1^{\prime\prime} \geq - \tau_0 \bar{w}_1^{\prime} +  \sigma_0 \bar{w}_1 - \bar{c} \bar{w}_1^{q} \geq \delta_0^2 \bar{w}_1 - \bar{c} \bar{w}_1^{q}~~~~~\textmd{in} ~ \mathbb{R}.
\end{equation}
If $\limsup_{t \to +\infty} \bar{w}_1(t)>0$, then there exists a sequence of local minimum points $t_i$ of $\bar{w}_1$ such that $t_i\to +\infty$ and $\bar{w}_1(t_i) \to 0$. By \eqref{3PST01}, there exists $\bar{\epsilon} >0$ small such that $\bar{w}_1^{\prime\prime}(t) >0$ whenever $\bar{w}_1(t) < 2\bar{\epsilon}$. Therefore, there exist $\bar{t}_i < t_i$ such that $\bar{w}_1(\bar{t}_i)=\bar{\epsilon}$ and $\bar{w}_1^{\prime}(t) < 0$ for $t\in [\bar{t}_i, t_i)$. By the same argument as that of Case 2 in  Theorem \ref{T02} (2), we  can get a contradiction. Hence $\lim_{t \to +\infty} \bar{w}_1(t)=0$. Then there exists $\bar{T}>0$ large such that $\bar{w}_1^\prime (t) < 0$ for $t\geq \bar{T}$. For any $\epsilon >0$ small, by \eqref{3PST01},  we can choose $\bar{T}$ large enough  such that
$$
\bar{w}_1^{\prime\prime}(t)  - (\delta_0 - \epsilon)^2\bar{w}_1(t)   \geq 0~~~~~~\textmd{for} ~t \geq \bar{T}.
$$
By a simple comparison principle argument,  we get
$$
\bar{w}_1(t) \leq \bar{w}_1(\bar{T}) \exp\{- (\delta_0 - \epsilon) (t- \bar{T})\} ~~~~~~\textmd{for} ~t \geq \bar{T}.
$$
From this we have that for any $\epsilon >0$, there exists $r_\epsilon >0$ and $c(\epsilon) >0$ such that $\bar{u}(x) \leq c(\epsilon) |x|^{-\epsilon}$ for $|x| < r_\epsilon$. Hence $\bar{u} \in L^\gamma (B_1)$ for all $\gamma >1$.  Recall that $n=3$, we have
$$
|x|^\alpha \bar{v}^{q+1} \leq C |x|^{\frac{p-1}{2} + \frac{2}{p-1} -4} ~~~~~~\textmd{for} ~0<|x| <1.
$$
Since $\frac{p-1}{2} + \frac{2}{p-1} -4 > -2$ by $3 < p< 5$,  we obtain $|x|^\alpha \bar{v}^{q+1} \in L^{\frac{3}{2- \theta_1}} (B_1)$ with some $\theta_1>0$ small.  We write  the first equation of \eqref{SPS-01} in the form $ -\Delta \bar{u} = g(x) \bar{u}$, where $g(x)=|x|^\alpha (\mu_1 \bar{u}^{2q} + \beta \bar{u}^{q-1} \bar{v}^{q+1} )(x) $, then $g \in L^{\frac{3}{2-\theta_2}}(B_1)$  for some $\theta_2>0$ small. By  Theorem 11 of Serrin \cite{S} , we obtain that  0 is a removable singularity of $\bar{u}$. In particular, $\bar{u} (x) \leq C$ for $|x|$ small, where $C$ is a positive constant. Hence we have $u(x) \leq C |x|^{-(n-2)}$ for $|x|$ large.

\vskip0.1in

If $\liminf_{t \to +\infty} \bar{w}_2(t)=0$, then $\lim_{t \to +\infty} (\bar{w}_1 + \bar{w}_2)(t)=0$, we easily obtain $\bar{\Psi} (+\infty; \bar{w}_1, \bar{w}_2)=0$, a contradiction. If $\liminf_{t \to +\infty} \bar{w}_2(t) > 0$, then $v(x)\geq C|x|^{-\frac{2}{p-1}}$ for $|x|$ large, a contradiction with the assumption. We complete the proof (4) in Theorem  \ref{T02}.

\end{document}